\documentclass[11pt]{amsart}

\usepackage{epigamath}

\usepackage[english]{babel}

\numberwithin{equation}{section}

\usepackage[shortlabels]{enumitem}

\usepackage{amsmath, amscd, amsfonts, amssymb, amsthm, latexsym, url, bm, framed, rotating}
\usepackage{mathtools}
\input{xy}
\xyoption{all}

\newtheorem{theorem}{Theorem}[section]
\newtheorem{lemma}[theorem]{Lemma}
\newtheorem{Key lemma}[theorem]{Key lemma}
\newtheorem{corollary}[theorem]{Corollary}
\newtheorem{proposition}[theorem]{Proposition}

\theoremstyle{definition}
\newtheorem{property}[theorem]{Property}
\newtheorem{definition}[theorem]{Definition}

\theoremstyle{remark}
\newtheorem{remark}[theorem]{Remark}
\newtheorem{example}[theorem]{Example}

\setlist[enumerate,1]{label={\rm(\arabic*)}, ref={\rm\arabic*}}

\newcommand{\bb}{\mathbb}

\newcommand{\comment}[1]{}

\renewcommand{\phi}{\varphi}

\newcommand{\roi}{\mathcal{O}}

\newcommand{\sub}[1]{{\mbox{\rm \scriptsize #1}}}

\newcommand{\To}{\longrightarrow}

\newcommand{\xto}{\xrightarrow}

\newcommand{\cal}{\mathcal}
\renewcommand{\hat}{\widehat}

\renewcommand{\tilde}{\widetilde}
\renewcommand{\ker}{\operatorname{Ker}}

\DeclareMathOperator{\Char}{char}
\DeclareMathOperator{\codim}{codim}

\DeclareMathOperator{\Frac}{Frac}

\DeclareMathOperator{\Hom}{Hom}
\DeclareMathOperator{\Id}{Id}
\DeclareMathOperator{\im}{im}
\DeclareMathOperator{\Jac}{Jac}

\DeclareMathOperator{\Pic}{Pic}

\DeclareMathOperator{\Spec}{Spec}

\newcommand{\CH}{C\!H}

\newcommand{\xTo}[1]{\stackrel{#1}{\To}}

\newcommand{\Nis}{\mathrm{Nis}}
\newcommand{\cdh}{\mathrm{cdh}}
\newcommand{\Gys}{\mathrm{Gys}}
\newcommand{\sm}{\mathrm{sm}}
\newcommand{\red}{\mathrm{red}}
\newcommand{\StR}{\mathrm{StR}}

\newcommand{\longtwoheadrightarrow}{\relbar\joinrel\twoheadrightarrow}

\newcommand{\longhookrightarrow}{\lhook\joinrel\longrightarrow}

\def\lowcong{\vbox to 0pt{\vss\hbox{$\scriptstyle\cong$}\vskip-1.2pt}}
    
\EpigaVolumeYear{10}{2026} \EpigaArticleNr{19} \ReceivedOn{September 9, 2022}
\InFinalFormOn{March 3, 2026}
\AcceptedOn{April 24, 2026}

\title{Bloch--Ogus theory for smooth and semi-stable schemes \\in mixed characteristic}
\titlemark{Bloch--Ogus theory for smooth and semi-stable schemes in mixed characteristic}

\author{Morten L\"uders}
\address{Ruprecht-Karls-Universit\"at Heidelberg, 
Institut f\"ur Mathematik, 
Im Neuenheimer Feld 205, 
69120 Heidelberg,
Germany}
\email{mlueders@mathi.uni-heidelberg.de}

\authormark{M.~L\"uders}

\AbstractInEnglish{We study Bloch--Ogus theory and the Gersten conjecture for homology theories with duality satisfying certain properties, in particular for \'etale cohomology with finite coefficients coprime to the residue characteristic of the base, for smooth and semi-stable schemes in mixed characteristic. We prove the Gersten conjecture in the smooth case and prove a special case in the semi-stable situation. As a corollary of the smooth case, we obtain the surjectivity of the Galois symbol map for arbitrary local rings over an excellent discrete valuation ring.}



\acknowledgement{During the preparation of this work the author was supported by the DFG Research Fellowship LU 2418/1-1.}

\begin{document}



\maketitle

\begin{prelims}

\DisplayAbstractInEnglish

\bigskip

\DisplayKeyWords

\medskip

\DisplayMSCclass

\end{prelims}


\newpage

\setcounter{tocdepth}{1}

\tableofcontents


\section{Introduction}
Let $S$ be the spectrum of an excellent discrete valuation ring $\cal O_K$ with local parameter $\pi$, field of fractions~$K$ and residue field $k$. 
For any scheme $X$ of finite type over $S$, one can define the concept of a homology theory with duality $(H_*,H^*)$, which is a pair of a homology and a cohomology theory satisfying certain axioms, for example Poincar\'e duality (see Section~\ref{section_homology_theories}). Let $\cal H^*(n)$ be the sheafification of the presheaf $U\mapsto H^*(U,n)$ on $X$ in the Zariski topology. For the sheaf $\cal H^*(n)$, one can formulate the Gersten conjecture, which says that the sequence of sheaves
\[0\To \cal H^q(n) \To \bigoplus_{x\in X^{(0)}}i_{*,x}H^{q}(x,n)\To \bigoplus_{x\in X^{(1)}}i_{*,x}H^{q-1}(x,n-1)\To\cdots\To \bigoplus_{x\in X^{(d)}}i_{*,x}H^{q-d}(x,n-d)\To 0 \]
is exact for all $n,q \in \bb N$ if $X$ is regular of dimension $d$. If this is the case, then we say that the Gersten conjecture holds for $\cal H$ in the Zariski topology.
Our main theorem reads as follows. 

\begin{theorem}[\textit{cf.}~Theorem~\ref{theorem_smooth_case}]\label{theorem_intro}
  Let $S$ and $X$ be as above and $(H_*,H^*)$ be a homology theory with duality satisfying the principal triviality property $($Property~\ref{property_local_principality}\,$)$ and lifting property $($Property~\ref{property_lifting}\,$)$. 
 If\, $X$ is smooth over $S$, then the Gersten conjecture holds for $\cal H$ in the Zariski topology.
\end{theorem}

The main example to which Theorem~\ref{theorem_intro} applies is that of \'etale cohomology with finite coefficients coprime to $\Char(k)$ (see Section~\ref{subsection_example_etale_cohom}). 
Combining Theorem~\ref{theorem_intro} for this example with the main result of \cite{Lu20}, we obtain the following corollary.

\begin{corollary}[Corollary~\ref{corollary_galois_symbol}]
Let $\cal O_K$ be an excellent discrete valuation ring with residue field $k$ and $R$ be a local $\cal O_K$-algebra. Let $\ell\in \bb N_{>0}$ be coprime to $\Char(k)$. Then the Galois symbol map
$$
\hat{K}^M_q(R)/\ell\longrightarrow H^q_\sub{\'et}\left(R,\mu_\ell^{\otimes q}\right)
$$
is surjective. In particular, the graded ring $\bigoplus_{q\in \bb N}H^q_\sub{\'et}(R,\mu_\ell^{\otimes q} )$ is generated by elements in degree $1$.
\end{corollary}

Our main theorem in the semi-stable situation is the following.

\begin{theorem}[Theorem~\ref{theorem_semi_stable}]\label{theorem_semi_stable_intro}
Let $\cal O_K$ be a henselian discrete valuation ring with local parameter $\pi$ and residue field $k$. Let $(H_*,H^*)$ be a homology theory with duality satisfying the blow-up property $($Property~\ref{property_blowup_formula}\,$)$, the principal triviality property $($Property~\ref{property_local_principality}\,$)$ and the lifting property $($Property~\ref{property_lifting}\,$)$. 
Let 
$$
X=\Spec(\roi_K[t_1,t_2]/(t_1^{} t_{2}^{}-\pi)).
$$
Let $\cal H^*(n)$ be the sheafification of the presheaf $U\mapsto H^*(U,n)$ on $X$ in the Zariski topology. Assume that for every open subset $W\subset X$ containing $\Spec(\roi_K[t_1,t_2]/(t_1^{}t_{2}^{}-\pi,t_1))\cong \bb A^1_k$, the morphism $ \bb A^1_k\to W$ is homologically effaceable at the origin. Then the complex of sheaves
\begin{equation*}
0\To \cal H^q(n) \To \bigoplus_{x\in X^{(0)}}i_{*,x}H^{q}(x,n)\To \bigoplus_{x\in X^{(1)}}i_{*,x}H^{q-1}(x,n-1)\To \bigoplus_{x\in X^{(2)}}i_{*,x}H^{q-2}(x,n-2)\To 0 
\end{equation*}
is exact.
\end{theorem}

Theorem~\ref{theorem_intro} has a long history.
In the case that $S$ is a field and $X$ is smooth over $S$, Theorem~\ref{theorem_intro} is due to Bloch and Ogus \cite{BO74}, who use ``Quillen's trick'' in their proof, an idea based on Noether normalisation which Quillen used in \cite{Qu72} to prove the Gersten conjecture for algebraic K-theory for smooth schemes over fields. The theorem of Bloch and Ogus has been studied further by Colliot-Th\'el\`ene, Hoobler and Kahn in \cite{CHK97}, where they give a different proof replacing ``Quillen's trick'' by Gabber's geometric presentation lemma which allows further applications. They therefore call the theorem the ``Bloch--Ogus--Gabber theorem''. The case of the Gersten conjecture for \'etale cohomology with finite coefficients coprime to $\Char(k)$ for a general regular scheme in equal characteristic can be reduced to the smooth case by a method developed by Panin (see \cite[Section~5]{Pa03}). In mixed characteristic, the case of $X$ smooth over $S$ has been proved in the Nisnevich topology, assuming additionally that $k$ is infinite, by Schmidt and Strunk for \'etale cohomology with finite coefficients coprime to $\Char(k)$ in \cite{SS19}, using their version of Gabber's geometric presentation lemma in mixed characteristic from \cite{SS18}. Theorem~\ref{theorem_intro} improves their result to the Zariski topology and arbitrary residue fields. The main difference is that we make use of ``Quillen's trick'' and follow the influential ideas of Gillet and Levine in their work on the Gersten conjecture for algebraic K-theory for smooth schemes over a discrete valuation ring \cite{GL87}, which are completely different from and predate the work of Panin. The general case of the Gersten conjecture for a regular scheme in mixed characteristic is wide open for \'etale cohomology with finite coefficients coprime to $\Char(k)$ but also for any other theory for which one would expect it to hold; examples include motivic cohomology, K-theory and Milnor K-theory. This mixed characteristic case would, however, have many applications, for example to the (contravariant) functoriality of motivic cohomology in the Chow group and higher zero-cycle range (see \cite{KEW16} and Remark~\ref{remark_introduction} below). Theorem~\ref{theorem_semi_stable_intro} is a first step in this direction. The method we develop in the proof, and which so far only works in relative dimension $1$, is to use blow-ups and an induction on the number of irreducible components of the special fiber. This is inspired by \cite[Section~3]{Hyodo1988}. 

\begin{remark}[Functoriality in the semi-stable case]\label{remark_introduction}
  The Gersten conjecture in the Zariski or Nisnevich topology for \'etale cohomology with finite coefficients coprime to $\Char(k)$ for a regular scheme $X$ flat over~$S$ as above with special fiber $X_1$ a simple normal crossing divisor would imply the Gersten conjecture for the improved Milnor K-sheaf $\hat{\cal K}^M_{q,X,\Nis}/m$, with $m$ coprime to $\Char(k)$, in the Nisnevich topology. Indeed, the sheaf $\hat{\cal K}^M_{q,X,\Nis}/m$ identifies with the sheaf $\cal H^q(\mu^{\otimes q}_m)$ for henselian local rings: Milnor K-theory is rigid by Proposition~\ref{proposition_rigidity_KM}, and \'etale cohomology is rigid by proper base change; passing to a bigger residue field and using a standard norm argument, one can identify the two. The other terms in the respective Gersten resolutions are identified by the norm residue isomorphism. This allows one to define a restriction map
$$
\CH^q(X)/m \To H^{2q}_\cdh(X_1,\bb Z/m\bb Z(q))
$$
through the diagram
\[
\xymatrix{
H^q_\Nis\left(X,\hat{\cal K}^M_{q}/m\right) \ar[r] & H^q_\Nis\left(X_1,\hat{\cal K}^M_{q}/m\right) \ar[d]^\cong \\
\CH^q(X)/m \ar@{-->}[r]  \ar[u]^\cong & H^{2q}_\cdh(X_1,\bb Z/m\bb Z(q))  ,
}
\]
in which the left vertical isomorphism, which is Bloch's formula, comes from the Gersten conjecture. For the isomorphism on the right,  see \cite[Section~8.2]{KEW16}. This restriction map is studied in \textit{loc.\ cit.}~for $q$ the relative dimension of $X$ over $S$.
\end{remark}

\subsection*{Acknowledgments}
The author would like to thank Matthew Morrow for many discussions on rigidity and Stefan Schreieder for helpful comments. 

\section{Homology theories}\label{section_homology_theories}

\subsection{Homology theories with duality}
We define homology theories with duality. The main references are \cite{Laumon1974,BO74,JS03}. 

\begin{definition}\label{definition_homology}
Let $S$ be the spectrum of a field or of an excellent Dedekind domain. Let  $\cal{C}_{S}$ be the category of schemes that are separated and of finite type over $S$. Let  $\cal{C}_{S_*}$ be the category with the same objects as $\cal{C}_{S}$ but where morphisms are just the proper maps in  $\cal{C}_{S}$. Let $Ab$ be the category of abelian groups. A \textit{homology theory} $H(n)=\{H_q(n)\}_{q\in \bb Z}$ on $\cal{C}_{S}$, for some fixed $n\in \bb Z$, is a sequence of covariant functors
$$H_q(-,n)\colon\cal{C}_{S_*}\To Ab$$
satisfying the following conditions:
\begin{enumerate}
\item (Presheaf in the \'etale topology) For each \'etale morphism $j\colon V\to X$ in $\cal{C}_{S}$, there is a map $j^*\colon H_q(X,n)\to H_q(V,n)$, associated to $V$ in a functorial way. 
\item (Localisation) If $i\colon Y\hookrightarrow X$ is a closed immersion with open complement $j\colon U\to X$, there is a long exact sequence (called the localization sequence) 
  $$\cdots\overset{\partial}\To H_q(Y,n)\xTo{i_*}
  H_q(X,n)\xTo{j^*} H_q(U,n)\xTo{\partial} H_{q-1}(Y,n)\To\cdots .$$
This sequence is functorial with respect to proper maps and open immersions in an obvious way.
\end{enumerate}
\end{definition}

\begin{remark}(Twist) The \textit{twist} $n$ in the above definition is fixed and could therefore be dropped from the notation. We keep it to keep track of the twist in the corresponding cohomology theory. \end{remark}

\begin{definition}\label{def.HTD}
A \textit{cohomology theory with supports} is given by abelian groups $H^q_Y(X,n)$, for any closed immersion $Y\hookrightarrow X$ of schemes in $\cal C_S$ and for any $q,n\in \bb Z$,
which satisfies the following conditions. For $Y=X$, the group $H^q_Y(X,n)$ is simply denoted by
$H^q(X,n)$. 
\begin{enumerate}
\item (Contravariant functoriality) The groups $H^q_Y(X,n)$ are contravariant functorial for Cartesian squares
\[
\xymatrix{
Y' \ar@{^{(}->}[r] \ar[d] & X' \ar[d]^f \\
Y \ar@{^{(}->}[r]  & X\rlap{;}
}
\]
\textit{i.e.}~there exists a pullback $f^*\colon H^q_Y(X,n) \to H^q_{Y'}(X',n)$.

\item\label{def.HTD-2} (Localisation)\label{pullback.locseq} For closed immersions $Z\hookrightarrow Y\hookrightarrow X$, there is a long exact sequence 
\[
\cdots \To H^q_Z(X,n) \To H^q_Y(X,n) \To H^q_{Y-Z}(X-Z,n) \stackrel{\partial}{\To} H^{q+1}_Z(X,n)\To \cdots
\]
compatible with pullbacks. Note that this implies in particular that
\[
H^q_Y(X,n) \simeq H^q_{Y_{\red}}(X,n)
\]
for a closed immersion $Y\hookrightarrow X$ and the reduced part $Y_{\red}$ of $Y$.

\item\label{def.HTD-3} (Excision) Let $Y\hookrightarrow X\in \cal C_S$ be a closed immersion and $U$ be an open subset of $X$ containing $Z$. Then there is an isomorphism 
  $$
  H^q_Z(X,n)\xTo{\lowcong}H^q_Z(U,n).
  $$
\end{enumerate}
\end{definition}

\begin{definition}\label{def.Hdual}
A \textit{homology theory with duality} $(H_*,H^*)$ is a tuple of a homology theory $H_*$ in the sense of Definition~\ref{definition_homology} and a cohomology theory with supports 
$H^*$ in the sense of Definition~\ref{def.HTD}
such that the following structural properties exist: 
\begin{enumerate}
\item\label{def.Hdual-1} (Cap product with supports) For any closed immersion $Y\hookrightarrow X$ in $\cal C_S$ and $p,q\in \bb Z$, there is a pairing 
$$\cap\colon H_p(X,m)\times H^q_Y(X,n)\To H_{p-q}(Y,m-n).$$
\item\label{def.Hdual-2} (Compatibility of cap product with restriction) Consider a Cartesian diagram
\[
\xymatrix{
Y' \ar@{^{(}->}[r] \ar[d]_{f_Y} & X' \ar[d]^{f_X} \\
Y \ar@{^{(}->}[r]  & X  
}
\]
such that $f_X$ is \'etale. Then for $a\in H^q_Y(X,n)$ and $z\in H_p(X,m)$, we have that $f_X^*(a)\cap f_X^*(z)=f_Y^*(a\cap z)$.  

\item\label{def.Hdual-3} (Projection formula) Consider a Cartesian diagram
\[
\xymatrix{
Y' \ar@{^{(}->}[r] \ar[d]_{f_Y} & X' \ar[d]^{f_X} \\
Y \ar@{^{(}->}[r]  & X  
}
\]
such that $f_X$ is proper. Let $a\in H^q_Y(X,n)$ and $z\in H_p(X',m)$. Then the diagram
\[
\xymatrix{
H_p(X',m)\ar@{}[r]|-{\times} \ar[d]_{f_{X*}} & H^q_{Y'}(X',n) \ar[r]^-{z\cap -}  & H_{p-q}(Y',m-n) \ar[d]^{f_{Y'*}} \\
H_p(X,m)\ar@{}[r]|-{\times}  & H^q_Y(X,n) \ar[r]^-{-\cap a} \ar[u]^{f_X^*}  & H_{p-q}(Y,m-n)  
}
\]
commutes in the sense that $f_{X*}(z)\cap a=f_{Y*}(z\cap f_{X}^*(a))$.

\item\label{def.Hdual-4} (Fundamental class) If $X\in \cal C_S$ is of dimension $d$ with structural morphism $\pi\colon X\to S$, then there is a global section $\eta_X$ of $H_{2d}(X,d-\dim \pi(X))$ such that if $\alpha\colon X'\to X$ is \'etale, then $\alpha^*\eta_X=\eta_{X'}$.

\item\label{def.Hdual-5} (Poincar\'e duality) If $\pi\colon X\to S\in \cal C_S$ is regular of dimension $d$ and $Y\hookrightarrow X$ is a closed immersion, then the cap product with $\eta_X$ induces an isomorphism
$$\eta_X\cap -\colon H_Y^{2d-q}(X,d-\dim \pi(X)-n)\To H_q(Y,n).$$
\end{enumerate}
\end{definition}

\begin{lemma}\label{lemma_compatibility_fundamental_class_etale_pullback}
Let $(H_*,H^*)$ be a homology theory with duality. Suppose that the square 
\[
\xymatrix{
Y' \ar@{^{(}->}[r] \ar[d]_{f_Y} & X' \ar[d]^{f_X} \\
Y \ar@{^{(}->}[r]  & X  
}
\]
is Cartesian and that $f_X$ is \'etale. Then the square
\[
\xymatrix{
H^q_{Y'}(X',n) \ar[r]^-{\cap \eta_{X'}}  & H_{2d-q}(Y',d-\dim \pi(X)-n) \\
H^q_Y(X,n) \ar[r]^-{\cap \eta_{X}} \ar[u]^{f^*_X}  & H_{2d-q}(Y,d-\dim \pi(X)-n) \ar[u]_{f_Y^*}  
}
\]
commutes. 
\end{lemma}

\begin{proof}
This follows from properties (\ref{def.Hdual-2}) and (\ref{def.Hdual-4}) in Definition~\ref{def.Hdual}.
\end{proof}

We will need the following properties. 

\begin{property}\label{property_blowup_formula}(Descent for blow-ups)
Let $X\in Ob(\cal C_S)$ and $\pi\colon \tilde X\to X$ be the blow-up of $X$ along a closed subscheme $Z$. Let $E\hookrightarrow \tilde X$ be the exceptional divisor. We say that a homology theory with duality $(H_*,H^*)$ satisfies the \textit{descent for blow-ups property} if for any integer $q$ and any $X,Z$ as above, there is a split exact sequence
$$
0 \To \tilde H_q(E,n)\To H_q\left(\tilde X,n\right)\xTo{\pi_*} H_q(X,n)\To 0,
$$
where
$$
\tilde H_q(E,n):=\ker\left(\pi_*\colon H_q(E,n)\to H_q(Z,n)\right).
$$
\end{property}

\begin{property}\label{property_local_principality}
(Principal triviality) Let $\pi\colon X\to S\in \cal C_S$. Let $i\colon W\hookrightarrow X$ be a smooth principal divisor and $\dim(W)=d$. Assume that $W$ is flat over $\pi(X)$. We say that a homology theory with duality $(H_*,H^*)$ satisfies the \textit{principal triviality property} if the morphism $i_*\colon H_{2d}(W,d-\dim \pi(X))\to H_{2d}(X,d-\dim \pi(X))$ is zero for all $X,W$ as above.
\end{property}

\begin{property}\label{property_lifting}
  (Lifting) Let $R$ be a regular semi-local ring of dimension $1$ and $Z$ be the set of closed points of $\Spec(R)$. We say that a homology theory with duality $(H_*,H^*)$ satisfies the \textit{lifting property} if the sequence
  $$
  0\To H^q(R,n)\To H^q(K(R),n)\To H^{q-1}(Z,n-1)\To 0
  $$
  is exact. 
\end{property}

\subsection{An example: \'Etale cohomology}\label{subsection_example_etale_cohom}
The following is the main example which satisfies the above definitions and properties. 

\begin{example}\label{PBexam3}
Let $S=\Spec (\cal O_K)$ be the spectrum of a  
discrete valuation ring $\cal O_K$ with residue field~$k$. 
Let $\Lambda=\bb Z/\ell^r\bb Z$, and let $\ell$ be a prime number coprime to $\Char(k)$. Let $\Lambda(n)=\mu_{\ell^r}^{\otimes n}$ for $n\in \bb N$ and  $\Lambda(n)=\Hom(\Lambda(-n),\Lambda)$ for $-n\in \bb N$ denote the Tate twists. One can define an \textit{\'etale homology theory with duality} as follows: define the homology of a morphism $\pi\colon X\to S\in \cal C_S$ to be
\[
H^\sub{\'et}_q(X,n) = H^{2-q}\left(X_\sub{\'et} , R\, \pi^! \Lambda (-n)\right) 
\]
and the cohomology with support in $Y\hookrightarrow X$ to be
\[
H_Y^q (X,n ) = H^{q}_Y\left(X_\sub{\'et} ,\Lambda (n)\right) .
\]
Indeed, it is shown in \cite[Example~2.2]{JS03}
and \cite[Proposition~1.5]{SS10} that $H^\sub{\'et}_q(-)$ defines a homology theory. The fact that $H_Y^q (X,n )$ is a  cohomology theory with supports is well known, and it is shown in \cite[Lemma~1.8]{SS10} that these two form a homology theory with duality. In particular, if $\dim X=d$ and $X$ is regular, then there is an isomorphism
$$\tau_{X,Y}\colon H_Y^{2d-q}(X,\Lambda(d-\dim(\pi(X))-n))\xTo{\lowcong} H_q(Y,\Lambda(n)).$$ 
For the convenience of the reader, we sketch the proof of this isomorphism if $\pi$ is flat, following the proof of the existence of this isomorphism of \textit{loc. cit.}. For the case in which $\pi$ factorises through the closed point of~$S$, \textit{i.e.}~is not flat, we refer to \cite[Section~2]{BO74} or also \cite[Lemma~1.8]{SS10}. Let
$$Tr_\pi\colon R\pi_!\Lambda(d-1+n)_X[2d-2]\To \Lambda_S(n)$$
be the trace morphism due to Deligne, see \cite[Theorem XVIII.2.9]{SGA4}, and
$$T_\pi\colon \Lambda_X(d-1+n)[2d-2]\To R\pi^!\Lambda_S(n)$$
its adjoint.
Applying $H^{2-q}_Y(X,-)$ to $T_\pi$ gives the morphism $\tau_{X,Y}$. It therefore suffices to show that $T_\pi$ is an isomorphism. Fix a locally closed embedding $\gamma\colon X\hookrightarrow \bb P:=\bb P^N_S$. Let $ \rho\colon\bb P\to S$ be the projection. Then $T_\pi$ factors as 
$$\Lambda_X(d-1+n)[2d-2]\xto{\Gys_\gamma} R\gamma^!\Lambda(N+n)_{\bb P}[2N] \xTo{T_\rho} R\gamma^!R\rho^! \Lambda_S(n)= R\pi^!\Lambda_S(n)$$
by \cite[Theorem 4.2.3.8(i)]{DeligneSGA}. Since the Gysin morphism $\Gys_\gamma$ is an isomorphism by absolute purity, see \cite{Fuj02}, and $T_\rho$ is an isomorphism by \cite[Theorem XVIII.3.2.5]{SGA4} (\textit{i.e.}~Poincar\'e duality for $\bb P$), we get that $T_\pi$ is an isomorphism.
\end{example}

We now show that the \'etale homology theory with duality, abusively called \'etale cohomology for short, of Example~\ref{PBexam3} has Properties~\ref{property_blowup_formula},~\ref{property_local_principality} and~\ref{property_lifting}.

\begin{lemma}[See also {\cite[Lemma~6.1]{SS10}}]
Let $X\in \cal C_S$ be regular of dimension $d$. Let $\pi\colon\tilde{X}\to X$ be the blow-up of\, $X$ along a closed subscheme $V\subset X$. Let $W$ be a closed subscheme of\, $X$ not contained in $V$ and $\tilde{W}$ its strict transform. Let $E_W\hookrightarrow \tilde{W}$ be the exceptional divisor. Then there is a split exact sequence
$$0\To \tilde{H}_q(E_W,n)\To H_q\left(\tilde{W},n\right)\To H_q(W,n)\To 0,$$
where
$$\tilde{H}_q(E_W,n):=\ker\left(\pi_*\colon H_q(E_W,n)\to H_q(V\cap W,n)\right).$$
In other words, the \'etale cohomology theory of Example~\ref{PBexam3} satisfies the blow-up property $($Property~\ref{property_blowup_formula}\,$)$.
\end{lemma}

\begin{proof}
There is a commutative diagram with exact sequences
\[
  \xymatrix{ 
 H_{q+1}(\tilde{W}-E_W,n) \ar[d]^\cong \ar[r] & H_q(E_W,n)\ar[r] \ar[d] & H_q\left(\tilde{W},n\right)\ar[r]  \ar[d]^{\pi_*}  & H_q\left(\tilde{W}-E_W,n\right) \ar[d]^\cong
  \\
H_{q+1}(W-V\cap W,n) \ar[r] & H_q(V\cap W,n)\ar[r] & H_q(W,n)   \ar[r]_{} & H_q(W-V\cap W,n)\rlap{.}
  }
\]
We need to show that the map $\pi_*\colon H_q(\tilde{W},n)\to H_q(W,n)$ is split surjective. We define a splitting $\pi^*\colon H_q(W,n)\to H_q(\tilde{W},n)$ of $\pi_*$ to be the composition
$$
\pi^*\colon H_q(W,n)\xTo{\lowcong} H_W^{2d-q}(X,d-\dim \pi(X)-n)\xTo{\pi^*} H_{\tilde{W}}^{2d-q}\left(\tilde{X},d-\dim \pi(X)-n\right)\xTo{\lowcong} H_q\left(\tilde{W},n\right)
$$
and show that $\pi_*\pi^*=\Id$. By the projection formula, we have a commutative diagram of cap-product pairings
\[
  \xymatrix{ 
H^0\left(\tilde{X},0\right) \ar@{}[r]|-{\times} \ar[d]^{\pi_*} &  H_{\tilde{W}}^{2d-q}\left(\tilde{X},d-\dim \pi(X)-n\right) \ar[r]^\cap  & H_{\tilde{W}}^{2d-q}\left(\tilde{X},d-\dim \pi(X)-n\right) \ar[d]^{\pi_*}    \\
H^{0}(X,0) \ar@{}[r]|-{\times}  &  H_W^{2d-q}(X,d-\dim \pi(X)-n) \ar[u]^{\pi^*} \ar[r]^{\cap} & H_W^{2d-q}(X,d-\dim \pi(X)-n)\rlap{.}
  }
\]
This implies that $\pi_*\pi^*(a)=\pi_*(1\cap \pi^*(a))=\pi_*(1)\cap a=a$ since $\pi_*(1)=1$ as $\pi$ is an isomorphism outside of~$E$.
\end{proof}

\begin{lemma}
The \'etale cohomology theory of Example~\ref{PBexam3} satisfies the principal triviality property $($Property~\ref{property_local_principality}\,$)$.
\end{lemma}

\begin{proof}
Let $\pi\colon X\to S\in \cal C_S$. Let $i\colon W\hookrightarrow X$ be a smooth principal divisor and $\dim(W)=d$. Assume that $W$ is flat over $\pi(X)$. Consider the commutative diagram 
\[
\xymatrix{
H_{2d}(W,d-\dim \pi(X)) \ar[r]^{}  &  H_{2d}(X,d-\dim \pi(X)) \\
H^0(W,\Lambda) \ar[r] \ar[u]_{\cong}  & H^{2}(X,\Lambda(1)) \ar[u]_{\cong}.
}
\]
The lower horizontal map sends the fundamental class of $W$ to zero since $[\cal O_W]=0\in \Pic(X)=\CH^1(X)$ as~$W$ is principal, which implies the statement. Indeed, this is the construction of the cycle class map to \'etale cohomology for codimension $1$ cycles, and the fact that it factorises through rational equivalence implies that it is zero for $W$ principal. For more details, we refer to \cite[Section~VI.9]{Mi80}.
\end{proof}

The following lemma is proved in \cite[Theorem B.2.1]{CHK97}. We give a sketch for the sake of completeness. Note that the case of discrete valuation rings is sufficient for applications to the Gersten conjecture in the smooth case (see the end of the proof of Corollary~\ref{theorem_smooth_case}), while we need the semi-local case in Section~\ref{section_semi_stable_case} for the semi-stable case.

\begin{lemma}\label{lemma_etalcoh_lifting_property}
The \'etale cohomology theory of Example~\ref{PBexam3} satisfies the lifting property $($Property~\ref{property_lifting}\,$)$.
\end{lemma}

\begin{proof}
We recall the notation: 
let $R$ be a regular semi-local ring of dimension $1$ and $Z$ be the set of closed points of $\Spec(R)$. We need to show that the sequence
$$
0\To H^q(R,n)\To H^q(K(R),n)\To H^{q-1}(Z,n-1)\To 0
$$
is exact. Let $R^h$ be the henselisation of $R$ along $Z$. Let $a\in H^{q-1}_\sub{\'et}(Z,\Lambda(n-1))$. By Gabber's affine analogue of proper base change, we have that $H^{q-1}_\sub{\'et}(Z,\Lambda(n-1))\cong H^{q-1}_\sub{\'et}(R^h,\Lambda(n-1))$. This implies that $a$ is in the image of some $a'\in H^{q-1}_\sub{\'et}(A,\Lambda(n-1))$, where $A$ is an \'etale neighbourhood of $Z$ and quasi-finite over $R$. Let~$R'$ be the integral closure of $R$ in $F'=\Frac(A)$.
Let $f\colon\Spec(R')\to \Spec(R)$ be the natural map, and recall that there is a trace map on cohomology induced by the counit map $f_!f^{*}\to \Id$ (see \cite[Expos\'e XVIII]{SGA4}). Let $T':=f^{-1}(Z)$.
We consider the commutative diagram
\[
\xymatrix{
H^q_\sub{\'et}(R',\Lambda(n)) \ar[r] \ar[d]^{tr} & H^q_\sub{\'et}(F',\Lambda(n)) \ar[r] \ar[d]^{tr} \ar@/^2pc/[rr]^-{\oplus\partial} & H^q_{T'}(R',\Lambda(n)) \ar[d]^{tr}  \ar@{=}[r]^-\sim &  \bigoplus_{\mathfrak{p}} H^{q-1}_\sub{\'et}(R'/\mathfrak{p},\Lambda(n-1)) \ar[d]^{tr} \\
H^q_\sub{\'et}(R,\Lambda(n)) \ar[r] & H^q_\sub{\'et}(\Frac(R),\Lambda(n)) \ar[r]^-{} \ar@/_2pc/[rr]_-{\partial}  & H^{q-1}_\sub{\'et}(Z,\Lambda(n-1)) \ar@{=}[r]^-\sim &  \bigoplus_{\mathfrak{m}}H^{q-1}_\sub{\'et}(R/\mathfrak{m},\Lambda(n-1)) ,
}
\]
where the $\mathfrak{p}$ are the prime ideals of $R'$ lying over $Z$ and the $\mathfrak{m}$ are the maximal ideals of $R$. By the Chinese remainder theorem, we can choose an element $g\in F'^\times$  which is congruent to $0$ at a copy of $Z$ in $T'$ and congruent to $1$ otherwise. Then $\partial\circ tr(a'\cup g)= tr\circ (\oplus\partial)(a'\cup g)=a$.
\end{proof}

\begin{remark}\label{remark_tate_twists}
  Let $S=\Spec(\cal O_K)$ be the spectrum of a discrete valuation ring $\cal O_K$ with perfect residue field~$k$ of characteristic $p>0$. 
Let $X$ be a regular semi-stable scheme over $S$, and let $\cal T_r(n)_X$ denote the $p$-adic \'etale Tate twists defined in \cite{Sa07}. One can define a \textit{$p$-adic \'etale homology theory} in weight $-1$ as follows (see \cite[Section~4.10]{JSS}): define the homology of a morphism $\pi\colon X\to S\in \cal C_S$ to be
\[
H_q(X,\bb Z/p^n(-1)) = H^{2-q}\left(X_\sub{\'et} , R\, \pi^! \cal T_r(1)_S\right) 
\]
and the cohomology with support in $Y\hookrightarrow X$ to be
\[
H_Y^q (X,n) = H^{q}_Y(X_\sub{\'et} ,\cal T_r(n)_X ) .
\]
There is a Poincar\'e duality isomorphism
\[
H_Y^{2d-q} (X,\cal T_r(d)_X) \cong H^{2-q}_Y(X_\sub{\'et} ,\cal T_r(d)_X[2d-2] )\cong H^{2-q}_Y\left(X_\sub{\'et} ,Rf^!\cal T_r(1)_S \right)\cong H_q(Y,-1),
\]
the relative duality theorem of Sato \cite[Theorem 7.3.1]{Sa07} giving the isomorphism in the middle. However, as the theory is not defined for arbitrary weights, it does not seem to be possible to define cap products with supports.
\end{remark}

\begin{remark}
While, as explained in Remark~\ref{remark_tate_twists}, being a homology theory with duality is a strong assumption, Properties~\ref{property_blowup_formula}--\ref{property_lifting} are more easily satisfied, which should make it possible to carry out the program in Section~\ref{section_semi_stable_case} for theories satisfying these as well. We just mention that motivic cohomology satisfies a blow-up formula (and so does K-theory for regularly immersed centers), has the principal triviality property due to the identification with higher Chow groups, and with finite coefficients has the lifting property for discrete valuation rings (dvrs); see \cite[Proof of Corollary~4.5]{Ge04} (and so does K-theory with finite coefficients, see for example \cite[Theorem V.6.9.1]{We13}, where an argument similar to the one in the proof of Lemma~\ref{lemma_etalcoh_lifting_property} is carried out in the dvr case).
\end{remark}

\section{The Gersten conjecture: The smooth case}
Let $S$ be the spectrum of an excellent discrete valuation ring $\cal O_K$ and $X$ be a regular and flat scheme over $S$. Let $(H_*,H^*)$ be a homology theory with duality, and let $\cal H^*(n)$ be the sheafification of the presheaf $U\mapsto H^*(U,n)$ for the Zariski topology on $X$. We begin by recalling some general reductions which imply that the sequence of sheaves in the Zariski topology on $X$
\begin{equation}\label{star}
0\To \cal H^q(n) \To \bigoplus_{x\in X^{(0)}}i_{*,x}H^{q}(x,n)\To \bigoplus_{x\in X^{(1)}}i_{*,x}H^{q-1}(x,n-1)\To\cdots\To \bigoplus_{x\in X^{(d)}}i_{*,x}H^{q-d}(x,n-d)\To 0 \tag{$*$}
\end{equation}
is exact (Section~\ref{section_general_reductions}). Then we prove the smooth case in Section~\ref{section_smooth_case} and  part of the semi-stable case in Section~\ref{section_semi_stable_case}.

\subsection{General reductions}\label{section_general_reductions}
In this section let $(H_*,H^*)$ be a homology theory with duality, and let $\cal H^*(n)$ be the sheafification of the presheaf $U\mapsto H^*(U,n)$ for the Zariski topology on $X$.

\begin{definition}\leavevmode
\begin{enumerate}
\item Let $X$ be a scheme, separated and of finite type over $S$. We let
  $$Z^p(X):=\{Z\subset X \mid Z\;\mathrm{ closed},\ \codim_X(Z)\geq p\}.$$
If $X$ is clear from the context, then we also write $Z^p$ instead of $Z^p(X)$. Let $Z^p/Z^{p+1}$ denote the ordered set of pairs $(Z, Z')\in Z^p\times Z^{p+1}$ such that $Z'\subset Z$, with the ordering $(Z,Z')\geq (Z_1,Z_1')$ if $Z\supset Z_1$ and $Z'\supset Z_1'$.
\item For the Zariski topology on $X$, we define $\cal H^q_{Z^p}(n)$ to be the sheaf associated to the presheaf
$$U\longmapsto \varinjlim_{Z\in Z^p(U)} H^q_{Z}(U,n)$$
and $\cal H^q_{Z^p/Z^{p+1}}(n)$ to be the sheaf associated to the presheaf
$$U\longmapsto \varinjlim_{(Z,Z')\in Z^p(U)/Z^{p+1}(U)} H^q_{Z-Z'}(U-Z',n).$$
\item Let
$$H_q(x,n):=\varinjlim_{U\subseteq \overline{\{x\}}} H_q(U,n),$$
$$H^q(x,n):=\varinjlim_{U\subseteq \overline{\{x\}}} H^q(U,n)$$
and
$$H^q_x(X,n):=\varinjlim_{U\ni x} H^q_{\overline{\{x\}}\cap U}(U,n),$$
where the colimit is taken over all open subschemes $U$ in $\overline{\{x\}}$ and all open subschemes $U$ in $X$ containing $x$, respectively.
\end{enumerate}
\end{definition}

The complex in \eqref{star} is exact if the natural map
$$(**)\colon\cal H^q_{Z^{p+1}}(n)\to \cal H^q_{Z^p}(n)$$ is zero for all $p,q,n\in \mathbb{N}$. Indeed, if $(**)=0$ for all $p,q,n\in \mathbb{N}$, the exact sequences 
$$\cdots\To \cal H^{q-1}_{Z^p/Z^{p+1}}(n)\To\cal H^q_{Z^{p+1}}(n)\xTo{0} \cal H^q_{Z^p}(n)\To \cal H^{q}_{Z^p/Z^{p+1}}(n)\To\cdots$$
and
$$\cdots\To \cal H^{q-1}_{Z^{p+1}/Z^{p+2}}(n)\To\cal H^q_{Z^{p+2}}(n)\xTo{0} \cal H^q_{Z^{p+1}}(n)\To \cal H^{q}_{Z^{p+1}/Z^{p+2}}(n)\To\cdots $$
splice together as follows:
$$\cdots\To \cal H^{q-1}_{Z^p/Z^{p+1}}(n)\longtwoheadrightarrow\cal H^q_{Z^{p+1}}(n)\longhookrightarrow \cal H^q_{Z^{p+1}/Z^{p+2}}(n) \longtwoheadrightarrow \cal H^q_{Z^{p+2}}(n)\To\cdots .$$
If $X$ is regular, the Poicar\'e duality isomorphism (Definition~\ref{def.Hdual}(5)) implies that
$$H^q_x(X,n)\cong H^{q-2p}(x,n-p)$$
for $\codim_X(\overline{\{x\}})=p$. This in turn implies\footnote{We use the fact that if $T_1,\ldots,T_r$ are pairwise disjoint closed subsets of $X$, then $\bigoplus_i H^q_{T_i}(X,A)\cong H^q_{\cup T_i}(X,A)$ (see \cite[Lemma  1.2.1]{CHK97}).}  that there are isomorphisms 
$$
\cal H^{q-1}_{Z^p/Z^{p+1}}(n)\cong \bigoplus_{x\in Z^p/Z^{p+1}}i_{*,x}H^{q-2p-1}(x,n-p)
$$
and
$$
\cal H^{q}_{Z^{p+1}/Z^{p+2}}(n)\cong \bigoplus_{x\in Z^{p+1}/Z^{p+2}}i_{*,x}H^{q-2p-2}(x,n-p-1).
$$

The following definition is used in applications to show that $(**)$ is zero.

\begin{definition}
Let $f\colon Z_1\to Z_2$ be a morphism in $\cal C_{S*}$ and $T\subseteq Z_2$ a finite set. We say that $f$ is \textit{homologically effaceable} at $T$ if and only if there is an open $U\subset Z_2$ containing $T$ such that the composition
$$
H_*(Z_1)\To H_*(Z_2)\To H_*(U)
$$
is zero.
\end{definition}

\begin{lemma}\label{lemma_effaceability}
If, for every $x\in X$, every open subset $x\in W\subset X$ and every $x\in Z_1\in Z^{p+1}(W)$, there exist a $Z_2\in Z^p(W)$ and a homologically effaceable morphism $f\colon Z_1\to Z_2$ at $x$, then $(**)$ is zero. In cohomological terms, this is equivalent to saying that the composition 
$$
H^*_{Z_1}(W)\To H^*_{Z_1\cap U}(U)\To H^*_{Z_2\cap U}(U)
$$
$($which factorises through $H^*_{Z_2}(W))$ is zero.
\end{lemma}

\subsection{A key lemma due to Bloch and Ogus}
In this section we transfer two lemmas of Bloch and Ogus to the relative situation. The proofs go through the same way, but we give them for the convenience of the reader~-- all we need are the smoothness assumptions on certain morphisms, which are relative notions. 

\begin{lemma}[\textit{cf.} {\cite[Lemma~5.1]{BO74}}]\label{lemma_cartesian_square}
Let 
\[
  \xymatrix{ 
  Z_2 \ar@{^{(}->}[r]^{i_2}  \ar[d]_g & X_2 \ar[d]^{f}  \\
   Z_1\ar@{^{(}->}[r]^{i_1} \ar[r] & X_1 
   }
\]
be a commutative square in $\cal C_S$ with $i_1,i_2$ closed immersions and $f,g$ smooth. Let $T\subset X_2$ be a finite set of points contained in an affine scheme. Then after replacing $Z_2$ and $X_2$ by Zariski neighbourhoods of\, $T$, there exists a closed subscheme $X_2'\subset X_2$ containing $Z_2$ such that the induced morphism $f\colon X_2'\to X_1$ is still smooth and the square
\[
  \xymatrix{ 
  Z_2 \ar@{^{(}->}[r]^{i_2}  \ar[d]_g & X_2' \ar[d]^{f'}  \\
   Z_1\ar@{^{(}->}[r]^{i_1}  & X_1 
   }
\]
is Cartesian.
\end{lemma}

\begin{proof}
Let $Y=f^{-1}(Z_1)\subset X_2$. Let $\cal I$ be the ideal defining $Z_2$ as a subscheme of $X_2$ and $\bar{\cal I}$ be the ideal defining $Z_2$ as a subscheme of $Y$: 
\[
  \xymatrix{ 
  Z_2 \ar@{^{(}->}[r]^{\bar{\cal I}} \ar@{^{(}->}@/^2pc/[rr]^{\cal I} \ar[dr]_{\sm} & Y \ar[d]^{\sm} \ar@{^{(}->}[r] & X_2 \ar[d] \\
   & Z_1 \ar@{^{(}->}[r] & X_1\rlap{.}
   }
\]
Then since $Y$ and $Z_2$ are smooth over $Z_1$, there is a short exact sequence 
$$0\To \bar{\cal I}/\bar{\cal I}^2\To \Omega^1_{Y/Z_1}\otimes_{\cal O_{Y}} \cal O_{Z_2}\To \Omega^1_{Z_2/Z_1}\To 0$$
of locally free sheaves on $Z_2$. 
Replacing $X_2$ if necessary by an open neighbourhood of $T$, we can find sections $f_1,\ldots,f_r$ of $\cal I$ whose images $\bar f_1,\ldots,\bar f_r$ in $\bar{\cal I}$ form a basis of $\bar{\cal I}/\bar{\cal I}^2$. Let $X_2'\subset X_2$ be the subscheme defined by $f_1=\cdots=f_r=0$. 

By base change we have that $\Omega^1_{Y/Z_1}\cong \Omega^1_{X_2/X_1}\otimes_{\cal O_{X_2}} \cal O_{Y}$, which implies that the differentials $df_1,\ldots,df_r$ in $\Omega^1_{X_2/X_1}$ are independent in some neighbourhood of $T$. This implies that after shrinking, $X_2'$ is smooth over~$X_1$ since $\dim\Omega^1_{X_2'/X_1}=\dim(\Omega^1_{X_2/X_1}/< df_1,\ldots,df_r>)\leq \dim X_2'-\dim X_1$. Now it follows from Nakayama's lemma that  $\bar f_1,\ldots,\bar f_r$ generate $\bar{\cal I}$ in some neighbourhood of $T$, and therefore $\cal I$ is generated by $f_1,\ldots,f_r$ and the sheaf of ideals $\cal I_Y$ defining $Y$ in $X_2$. This means that the lower square in the statement of the lemma is Cartesian; indeed, taking a quotient by $f_1,\ldots,f_r$ identifies $Z_2$ and $Y$ in $X_2'$. 
\end{proof}

\begin{Key lemma}[\textit{cf.} {\cite[Proposition~4.5]{BO74}}]\label{key_lemma_smooth_case}
Let $S$ be the spectrum of an excellent discrete valuation ring $\cal O_K$. Let $(H_*,H^*)$ be a homology theory with duality on $\cal C_S$ satisfying the principal triviality property $($Property~\ref{property_local_principality}\,$)$. 
Let $i\colon Z_1\hookrightarrow Z_2$ be a closed immersion of affine schemes in $\cal C_S$ and $\pi\colon Z_2\to Z_1$ be a section of\, $i$ which is smooth of relative dimension~$1$ at a finite set $T$ of points in $Z_2$. Then $i$ is homologically effaceable at $T$.
\end{Key lemma}

\begin{proof}
We can assume that $Z_1$ is flat over $S$. Otherwise, we are in the situation of \cite[Proposition~4.5]{BO74}. Since $Z_1$ and $Z_2$ are affine, we can find a commutative diagram
\[
  \xymatrix{ 
  Z_2 \ar@{^{(}->}[r]^{i_2}  \ar[d]_\pi & X_2 \ar[d]^{f}  \\
   Z_1\ar@{^{(}->}[r]^{i_1} \ar[r] \ar@/^1pc/[u]^i & X_1 
   }
\]
with $X_1$ and $X_2$ smooth over $S$, $f$ smooth and $i_1$ and $i_2$ closed immersions. By Lemma~\ref{lemma_cartesian_square} we can find a neighbourhood $U$ of $T$ in $X_2$ and a closed subscheme $X_2'\subset U\cap X_2$ containing $Z_2'=Z_2\cap U$ such that $f'\colon X_2'\to X_1$ is smooth of relative dimension $1$, $T\subset Z_2'$ and such that the square
\[
  \xymatrix{ 
  Z_2' \ar@{^{(}->}[r]^{i_2}  \ar[d]_{\pi'} & X_2' \ar[d]^{f'}  \\
   Z_1\ar@{^{(}->}[r]^{i_1} \ar[r]  & X_1 
   }
\]
is Cartesian. Let $Z_1'=i^{-1}(Z_2')$. Then the map $\alpha\colon Z_1'\to Z_2'$ induced by $\pi'$ is an open immersion, and we get a commutative square 
\begin{equation}\label{com_diag_1}
  \xymatrix{ 
  Z_1' \ar@{^{(}->}[r]^{i'} \ar[rd]_\alpha & Z_2' \ar@{^{(}->}[r]^{i_2}  \ar[d]_{\pi'} & X_2' \ar[d]^{f'}  \\
   & Z_1\ar@{^{(}->}[r]^{i_1} \ar[r]  & X_1\rlap{.} 
   }
\end{equation}
Applying Lemma~\ref{lemma_cartesian_square} to the outer square of Diagram (\ref{com_diag_1}), we get a commutative diagram
\[
  \xymatrix{ 
  Z_1' \ar@{^{(}->}[r]^{}  \ar[d]_{\alpha} &W \ar@{^{(}->}[r]^{h} \ar[d]_{g} &  X_2' \ar[dl]^{f'}  \\
   Z_1\ar@{^{(}->}[r]^{} \ar[r]& X_1 &
   }
\]
in which the square on the left is Cartesian. Since $f'$ and $g$  are smooth, $h$ is a regular immersion, and since~$g$ is of relative dimension $0$, this implies that $W$, also smooth over $S$, is a smooth Cartier divisor on~$X'_2$. Shrinking $X_2'$ further around $T$, we may assume that $W$ is a principal Cartier divisor.

Putting these diagrams together, we get a commutative diagram
\[
  \xymatrix{ 
  Z_1' \ar[d] \ar@/_2pc/[dd]_\alpha \ar[r] & W \ar[d]^h \ar@/^2pc/[dd]^g  \\
  Z_2' \ar@{^{(}->}[r]^{i_2}  \ar[d]_{} & X_2' \ar[d]^{f'}  \\
   Z_1\ar@{^{(}->}[r]^{i_1} \ar[r]  & X_1 
   }
\]
in which $\alpha$ is an open immersion and $g$ is \'etale.

Let $d=\dim X_1=\dim W$ (and therefore $\dim X_2'=d+1$). Consider the commutative diagram
\[
  \xymatrix{ 
  H^q_{Z_1}(X_1,n) \ar[rr]^{f^*} \ar[rd]^{g^*} \ar[d]^{\cong}_{\cap \eta_{X_1}} & &   H^q_{Z_2'}\left(X_2',n\right) \ar[dl]_{h^*} \ar[d]^{\cap h_*\eta_W}  \\
   H_{2d-q}(Z_1,d-1-n)\ar[rd]^{\alpha^*} & H^q_{Z_1'}(W,n) \ar[d]^{\cap \eta_W} & H_{2d-q}\left(Z_2',d-1-n\right) \\
   & H_{2d-q}\left(Z_1',d-1-n\right) \ar[ur]^{i_*'}  &
   }
\]
in which the bottom-left square commutes by Lemma~\ref{lemma_compatibility_fundamental_class_etale_pullback} and the triangle at the top commutes since $g=h\circ f$. The bottom-right square commutes by the projection formula.

Now $\cap \eta_{X_1}$ is an isomorphism by Poincar\'e duality and $\cap h_*\eta_W=\cap 0=0$. Indeed, the right vertical map is given by $H^q_{Z_2'}(X_2',n)\times H_{2d}(X_2',d-1)\xto{-\cap h_*\eta_W} H_{2d-q}(Z_2',d-1-n)$, and the image of $\eta_W$ under the map $h_*\colon H_{2d}(W,d-1) \to H_{2d}(X_2',d-1)$ is zero by Property~\ref{property_local_principality} since $W$ is a smooth and locally principal divisor. This and the commutativity imply that $i_*'\alpha^*=0$. But now the commutative diagram
\[
  \xymatrix{ 
  Z_1\cap Z_2'=Z_1' \ar@{^{(}->}[r]^-{\alpha}  \ar[d]_{i'} & Z_1 \ar[d]^{i}  \\
   Z_2\cap U= Z_2' \ar@{^{(}->}[r]^-{\beta} \ar[r]  & Z_2, 
   }
\]
in which $\alpha$ and $\beta$ are open, implies that the diagram
\[
  \xymatrix{ 
  H_{2d-q}\left(Z_1',d-n\right) \ar[d]_{i_*'} & H_{2d-q}(Z_1,d-n) \ar[d]^{i_*} \ar[l]_{\alpha^*}  \\
   H_{2d-q}\left(Z_2',d-n\right)   & H_{2d-q}(Z_2,d-n) \ar[l]^{\beta^*}
   }
\]
commutes and therefore that the composition
$$
H_{2d-q}(Z_1,d-n)\xTo{i_*}H_{2d-q}(Z_2,d-n)\xTo{\beta^*}H_{2d-q}\left(Z_2',d-n\right)
$$
is zero, which means that $i$ is homologically effaceable at $T$. 
\end{proof}

\subsection{Proof of the smooth case}\label{section_smooth_case}

\begin{theorem}\label{theorem_smooth_case}
Let $S=\Spec(\cal O_K)$ be the spectrum of an excellent discrete valuation ring $\cal O_K$ and $X$ be a smooth scheme over $S$. Let $(H_*,H^*)$ be a homology theory with duality on $\cal C_S$ satisfying the principal triviality property $($Property~\ref{property_local_principality}\,$)$ and the lifting property $($Property~\ref{property_lifting}\,$)$. Let $\cal H^*(n)$ be the sheafification of the presheaf\, $U\mapsto H^*(U,n)$ for the Zariski topology on $X$.
Then the sequence of sheaves
\[0\To \cal H^q(n) \To \bigoplus_{x\in X^{(0)}}i_{*,x}H^{q}(x,n)\To \bigoplus_{x\in X^{(1)}}i_{*,x}H^{q-1}(x,n-1)\To\cdots\To \bigoplus_{x\in X^{(d)}}i_{*,x}H^{q-d}(x,n-d)\To 0 \]
is exact.
\end{theorem}

\begin{proof}
Let $x\in X$. We need to show that the map $(**)$ is zero at $x$. For this we are going to use Lemma~\ref{lemma_effaceability}. In order to simplify the notation, let $X$ be an open neighbourhood of $x$. Let $Z'\in Z^{p+1}(X)$, $x\in Z'$; assume that either $Z'$ is flat, or $p>0$. Then up to shrinking $X$, by \cite[Lemma~1]{GL87}, there exists a smooth morphism of relative dimension $1$
$$
g\colon X\To A
$$
such that the restriction $f:=g|_{Z'}\colon Z'\to A$ is quasi-finite. Note that if $Z'$ is not flat over $\roi_K$ but of codimension greater than $1$, then \cite[Lemma~1]{GL87} still applies since $Z'$ is contained in a flat divisor of~$X$. Let $Z'':=Z'\times_A X$ and $i\colon Z'\to Z''$ be the induced closed immersion. By Zariski's main theorem, $f'$ factorises into an open immersion $Z''\to \bar{Z}''$ and a finite map $\bar{f}'\colon \bar{Z}''\to X$.  There is commutative diagram
\[
  \xymatrix{ 
  Z' \ar@{^{(}->}[rrd] \ar[ddr]  \ar[dr]^i & & \bar{Z}'' \ar[d]^{\bar f'}  \\
   & Z''\ar[r]^{f'} \ar[d]^{g'} \ar[ru] & X \ar[d]^g
  \\
 & Z'   \ar[r]_{f} & A\rlap{.}
  }
\]
Since $g$ is smooth at $x$ and of relative dimension $1$, $g'$ is smooth of relative dimension $1$ at the points of $S'=f'^{-1}(x)$. Furthermore, since $\bar f'$ is finite, $Z=\bar f'(\bar{Z}'')\in Z^p(X)$ and $Z'\subset Z$. By Lemma~\ref{key_lemma_smooth_case}, $i$ is homologically effaceable at $x$; that is, there is an open neighbourhood $U''$ of $S'$ in $Z''$ such that the composition
$$H_q(Z',n)\To H_q(Z'',n)\To H_q(U'',n)$$
is zero. This implies that the composition
$$H_q(Z',n)\To H_q(\bar Z'',n)\To H_q(U'',n)$$
is zero. Since $\bar f'$ is finite, we can find a neighbourhood $U$ of $x$ in $X$ such that $U'=\bar f'^{-1}(U)\subset U''$. The commutative diagram
\[
  \xymatrix{ 
   H_q(Z',n) \ar[r] \ar@{=}[d] \ar@/^2pc/[rr]^0 & H_q\left(\bar Z'',n\right)\ar[r]^{} \ar[d]^{\bar f'_*}  & H_q(U'',n) \ar[r] & H_q(U',n) \ar[dl]^{\bar f'_*}
  \\
 H_q(Z',n) \ar[r] & H_q(Z,n)   \ar[r]_{} & H_q(U,n) &
  }
\]
implies the effaceability of the map $Z'\to Z$ for subschemes $Z'$ which are flat over $\cal O_K$ or of codimension greater than $1$. This implies that the map
$$(**)\colon\cal H^q_{Z^{p+1}}(n)\To \cal H^q_{Z^p}(n)$$
is zero for all $p>0$ and arbitrary $q$.

It remains to show that
$$(**)\colon\cal H^q_{Z^{1}}(n)\To \cal H^q_{Z^0}(n)$$
is zero. This map fits into the exact sequence $\cal H^q_{Z^{1}}(n)\to \cal H^q_{Z^0}(n)\to \cal H^q_{Z^0/Z^1}(n)$ (see Section~\ref{section_general_reductions}). By definition $\cal H^q_{Z^0}(n)_x\cong H^q_{}(\roi_{X,x},n)$ and $\cal H^q_{Z^0/Z^1}(n)_x\cong H^q_{}(K(X),n)$. We can therefore also show that the composition  $ H^q_{}(\roi_{X,x},n)\to H^q_{}(\roi_{X,\eta},n) \to H^q_{}(K(X),n)$ is injective, where $\eta$ is the generic point of the special fiber of~$X$. The first map is obtained by localising the exact sequence $\cal H^q_{Z^1-\eta}(n)\to \cal H^q_{Z^0}(n) \to \cal H^q_{Z^0/(Z^1-\eta)}(n)$ at $x$, and by what we showed above, the first map in this sequence is zero.  
The map $H^q_{}(\roi_{X,\eta},n) \to H^q_{}(K(X),n)$ fits into a long exact sequence
$$\cdots\To H^{q+1}_{}(K(X),n+1)\xTo{\partial} H^{q}_{}(k(\eta),n) \To H^q_{}(\roi_{X,\eta},n) \To H^q_{}(K(X),n)\xTo{\partial}  H^{q-1}_{}(k(\eta),n-1)\To\cdots .$$ 
As $\partial$ is surjective by the lifting property (Property~\ref{property_lifting}), this concludes the proof.
\end{proof}

\subsection{An application to the Galois symbol map}
In order to extend results on the Galois symbol from smooth to singular local rings, we need the notion of a henselian surjection.
That a ring map $A\to R$ is a henselian surjection means that it is surjective and that $(A,I=\ker(A\to R))$ is a henselian pair. The definition of such a pair, which we will use, is the following. 

\begin{definition}  Let $A$ be a ring and $I\subset A$ an ideal and $(A,I)$ the corresponding pair. Then $(A,I)$ is a henselian pair if it satisfies the following conditions:
\begin{enumerate}[(i)]
\item $I$ is contained in the Jacobson radical of $A$, and
\item given a monic polynomial $f(x)\in A[x]$ and a root $\bar\alpha\in A/I$ of $\bar f\in (A/I)[x]$ with $\bar f'(\bar \alpha)\in (A/I)^\times$, then $\bar\alpha$ lifts to a root $\alpha\in A$ of $f$.
\end{enumerate}
\end{definition}

For more details on henselian pairs, see \cite[Tag 09XD]{stacks-project}.

\begin{corollary}\label{corollary_galois_symbol}
Let $\cal O_K$ be an excellent discrete valuation ring with residue field $k$ and $R$ be a local $\cal O_K$-algebra. Let $\ell\in \bb N_{>0}$ be coprime to $\Char(k)$. Then the Galois symbol map
$$\hat{K}^M_q(R)/\ell\longrightarrow H^q_\sub{\'et}\left(R,\mu_\ell^{\otimes q} \right)$$
is surjective. In particular, the graded ring
$$\bigoplus_{q\in \bb N}H^q_\sub{\'et}\left(R,\mu_\ell^{\otimes q} \right)$$ is generated by elements in degree $1$.
\end{corollary}

\begin{proof}
First assume  that $R$ is smooth. Then the first assertion follows from the commutative diagram with exact rows
\[
  \xymatrix{ 
 & \hat{K}^M_q(R)/\ell  \ar[r] \ar[d]   & {K}^M_q(\Frac(R))/\ell \ar[r]^{} \ar[d]^{\cong}  & \ar[d]^{\cong} \bigoplus_{x\in\Spec(R)^{(1)}} {K}^M_{q-1}(R)/\ell \ar[r] &  \cdots
  \\
0  \ar[r] & H^q_\sub{\'et}\left(R,\mu_\ell^{\otimes q} \right) \ar[r] & H^q_\sub{\'et}\left(\Frac(R),\mu_\ell^{\otimes q} \right)  \ar[r]_{} & \bigoplus_{x\in\Spec(R)^{(1)}} H^{q-1}_\sub{\'et}\left(k(x),\mu_\ell^{\otimes q-1} \right) \ar[r]_{}  & \cdots\rlap{.}
  }
\]
The first row is exact by \cite[Theorem 1.1]{Lu20}. Indeed, the proofs of \textit{loc. cit.} go through in exactly the same way with finite coefficients. The second row is exact by Corollary~\ref{theorem_smooth_case}. The vertical arrows from the second one onwards are isomorphisms by the Bloch--Kato conjecture proven by Rost and Voevodsky. This implies that the first vertical map is surjective. 

Now let  $R$ be an arbitrary local $\cal O_K$-algebra, and let $A\to R$ be a henselian surjection from an ind-smooth local $\cal O_K$-algebra $A$. Consider the commutative diagram
\[
  \xymatrix{ 
  & \hat{K}^M_q(A)/\ell \ar[r]^{} \ar@{->>}[d]  & \ar[d] \hat{K}^M_q(R)/\ell   
  \\
 & H^q_\sub{\'et}\left(A,\mu_\ell^{\otimes q} \right)  \ar[r]^-{\cong} & H^q_\sub{\'et}\left(R,\mu_\ell^{\otimes q} \right)\rlap{.}
  }
\]
The left  vertical map is surjective by the last paragraph and since both \'etale cohomology and improved Milnor K-theory commute with filtered colimits of rings. The bottom map is an isomorphism by Gabber's affine analogue of the proper base change theorem; see \cite{Gabber1994}. In sum, we get that the map $\hat{K}^M_q(R)/\ell\to H^q_\sub{\'et}(R,\mu_\ell^{\otimes q})$ is surjective.

Since the map ${K}^M_q(R)/\ell\to \hat{K}^M_q(R)/\ell$ is surjective by \cite[Theorem B]{Ke10}, the composition ${K}^M_q(R)/\ell\to \hat{K}^M_q(R)/\ell\to H^q_\sub{\'et}(R,\mu_\ell^{\otimes q} ) $ is surjective, implying the second assertion.
\end{proof}

We include the following proposition, which is stronger than the surjectivity needed in the previous proof, for the convenience of the reader. It says that Milnor K-theory with finite invertible coefficients is rigid (for the precise meaning of rigid, see the proposition). This is probably well known to  experts and generalises \cite[Lemma~3.2]{Da18}, whose proof we follow closely.

\begin{proposition}\label{proposition_rigidity_KM}
Let $k$ be a commutative ring. Let $R$ be an arbitrary local $k$-algebra, and let $A\to R$ be a henselian surjection from an ind-smooth local $k$-algebra $A$. Let $\ell\ge 0$ be an integer invertible in $R$ and $A$. Then the natural map
$${K}^M_q(A)/\ell\To {K}^M_q(R)/\ell$$
is an isomorphism.
\end{proposition}

\begin{proof}
The statement is trivial for $q=0$. For $q=1$ it follows from Hensel's lemma that $A^\times/\ell\xto{\lowcong} R^\times/\ell$. Let $q\geq 2$. There is a commutative diagram with exact rows
\[
  \xymatrix{ 
0 \ar[r] &  \StR_qA^\times \ar[r] \ar[d]   & \mathrm{T}_qA^\times \ar[r]^{} \ar[d]^{\cong}  & K^M_q(A)\ar[d]^{}  \ar[r] & 0 
  \\
0  \ar[r] & \StR_qR^\times  \ar[r] & \mathrm{T}_qR^\times  \ar[r]_{} &   K^M_q(R)  \ar[r]_{}  &  0\rlap{,}
  }
\]
where $\mathrm{T}_qA^\times= (A^\times)^{\otimes q} $ and $\mathrm{T}_qR^\times= (R^\times)^{\otimes q} $, and  
$\StR_qA^\times=\langle x_1\otimes\dots\otimes x_q\in \mathrm{T}_qA^\times\mid \exists i\neq j: x_i+x_j=1\rangle $ and $\StR_qR^\times=\langle x_1\otimes\dots\otimes x_q\in \mathrm{T}_qR^\times\mid \exists i\neq j: x_i+x_j=1\rangle $ are the Steinberg groups.
Tensoring with $\bb Z/\ell$, we get the commutative diagram with exact rows
\[
  \xymatrix{ 
   \StR_qA^\times/\ell \ar[r] \ar[d]^\alpha   & \mathrm{T}_qA^\times/\ell \ar[r]^{} \ar[d]^{\cong}  & K^M_q(A)/\ell\ar[d]^{}  \ar[r] & 0 
  \\
  \StR_qR^\times/\ell  \ar[r] & \mathrm{T}_qR^\times/\ell  \ar[r]_{} &   K^M_q(R)/\ell  \ar[r]_{}  &  0\rlap{,}
  }
\]
in which the middle vertical map is an isomorphism since $\mathrm{T}_qA^\times/\ell\cong \mathrm{T}_q(A^\times/\ell)\cong \mathrm{T}_q(R^\times/\ell)\cong \mathrm{T}_qR^\times/\ell$. Therefore, by the five-lemma, it suffices to show that $\alpha$ is surjective. Let $r_1\otimes\dots\otimes r_q\in \StR_qR^\times$. Without loss of generality, we may assume that $q=2$ and $r_1+r_2=1$. Let $a_1,a_2\in A^\times$ map to $r_1$ and $r_2$, respectively. Then $a_1+a_2=:u\in 1+I$, where $I=\ker(A\to R)$. Note that $I\subset \Jac(A)$ implies that $1+I\subset A^\times$. Therefore, $u^{-1}(a_1+a_2)= u^{-1}a_1+u^{-1}a_2=1$, and $u^{-1}a_1\otimes u^{-1}a_2\in \StR_2A^\times$ maps to $r_1\otimes r_2\in \StR_2R^\times$ under $\alpha$.
\end{proof}

\section{The Gersten conjecture: The semi-stable case}\label{section_semi_stable_case}

In this section we study the Gersten conjecture in the semi-stable case \'etale locally. \'Etale locally a \textit{semi-stable} scheme is of the form 
$$\Spec(\roi_K[t_1,\ldots,t_d]/(t_1\cdots t_{e}-\pi))$$
for $e\in \bb Z$ and $1\leq e\leq d$.

We will use the following notation:
\begin{itemize}
\item If $X$ is a Noetherian scheme and $V$ a closed subscheme, then we denote the blow-up of $X$ along $V$ by $\tilde{X}$.
\item If $W$ is a closed subscheme of $X$, then we denote the strict transform of $W$ by $\tilde{W}$.
\item We denote the exceptional divisors of $\pi_X\colon\tilde{X}\to X$ and $\pi_W\colon\tilde{W}\to W$ by $E_X$ and $E_W$, respectively. 
\end{itemize}

Before we state the main theorem of this section, we outline an idea for an inductive approach to the Gersten conjecture for schemes of the form $\Spec(\roi_K[t_1,\ldots,t_d]/(t_1\cdots t_{e}-\pi))$. We expect that there exists an analogue of Gabber's representation lemma for strictly semi-stable schemes in mixed characteristic which could reduce the Gersten conjecture for strictly semi-stable schemes to that for schemes of the above form. We propose to carry out an induction on $e$ starting with the smooth case of the last section as the induction beginning. Let $R=\roi_K[t_1,\ldots,t_d]/(t_1\cdots t_{e-1}-\pi)$ and $X=\Spec(R)$. Let $\pi_X\colon\tilde{X}\to X$ be the blow-up of $X$ along $V(I)$ for $I=t_{e-1}R+t_eR$. Then $\tilde{X}$ can be covered by the open sets
$$U_1=\Spec(\roi_K[t_1,\ldots,t_d,V]/(t_1\cdots t_{e-1}-\pi,t_eV-t_{e-1}))$$
and
$$U_2=\Spec(\roi_K[t_1,\ldots,t_d,W]/(t_1\cdots t_{e-1}-\pi,t_e-t_{e-1}W)).$$
Note that $U_1$ is exactly the object we want to study (for this blow-up, see also \cite[Section~3]{Hyodo1988}). Therefore, if we show that the complex of sheaves (\ref{equation_Gersten_complex}) on $\tilde{X}$ is exact at all points $x\in U_1$, then we are done. In the induction step, we propose to firstly use the blow-up property (Property~\ref{property_blowup_formula}) and secondly to proceed as in the smooth case, \textit{i.e.}~to make a reduction to the special fiber by beginning with the homological effaceability of subschemes which are flat over the base. This way we end up in a situation which is similar to the dvr case in the smooth case, for which we need the lifting property (Property~\ref{property_lifting}) (see the end of the proof of Theorem~\ref{theorem_semi_stable}).
In the proof of the following theorem, we carry out this program in the simplest situation (\textit{i.e.}~$d=2$).
In this section we drop the twist in the homology theory from notation. 

\begin{theorem}\label{theorem_semi_stable}
Let $\cal O_K$ be a henselian discrete valuation ring with local parameter $\pi$ and residue field $k$. Let $(H_*,H^*)$ be a homology theory with duality satisfying the blow-up property $($Property~\ref{property_blowup_formula}\,$)$, the principal triviality property $($Property~\ref{property_local_principality}\,$)$ and the lifting Property~\ref{property_lifting}. 
Let 
$$X=\Spec(\roi_K[t_1,t_2]/(t_1^{}t_{2}^{}-\pi)).$$
Let $\cal H^*(n)$ be the sheafification of the presheaf\, $U\mapsto H^*(U,n)$ on $X$ in the Zariski topology. Assume that for every open subset $W\subset X$ containing $\Spec(\roi_K[t_1,t_2]/(t_1^{}t_{2}^{}-\pi,t_1))\cong \bb A^1_k$, the morphism $ \bb A^1_k\to W$ is homologically effaceable at the origin. Then the complex of sheaves
\begin{equation}\label{equation_Gersten_complex}
0\To \cal H^q(n) \To \bigoplus_{x\in X^{(0)}}i_{*,x}H^{q}(x,n)\xTo{\partial_1} \bigoplus_{x\in X^{(1)}}i_{*,x}H^{q-1}(x,n-1)\xTo{\partial_0} \bigoplus_{x\in X^{(2)}}i_{*,x}H^{q-2}(x,n-2)\To 0 
\end{equation}
is exact.
\end{theorem}

\begin{proof}
The surjectivity of ${\partial_0}$ follows from the fact that for every $x\in X^{(2)}$ we can choose a $y\in X^{(1)}$ such that $\overline{\{y\}}$ is regular and contains $x$. The boundary map $\partial_0\colon H^{q-1}(y,n-1)\to H^{q-2}(x,n-2)$ is then surjective by the lifting property (Property~\ref{property_lifting}). This implies in particular that the map $(**)\colon\cal H^q_{Z^{2}}(n)\to \cal H^q_{Z^{1}}(n)$ is zero.

In order to show the exactness at the other places, we show that the map $(**)\colon\cal H^q_{Z^{1}}(n)\to \cal H^q_{Z^{0}}(n)$ is zero and proceed by the induction explained before the theorem. The case $e=1$ was proved in Corollary~\ref{theorem_smooth_case}. 
We change the notation slightly from that above. 
From now on, let $X=\Spec(\roi_K[t_1,t_2]/(t_1-\pi))$. Let $\pi_X\colon\tilde{X}\to X$ be the blow-up of $X$ along $V(I)$ for $I=t_{1}R+t_2R$. Then $\tilde{X}$ can be covered by the open sets
$$U_1=\Spec(\roi_K[t_1,t_2,V]/(t_1-\pi,t_2V-t_{1}))=\Spec(\roi_K[t_1,t_2]/(t_1^{}t_{2}^{}-\pi))$$
and
$$U_2=\Spec(\roi_K[t_1,t_2,W]/(t_1-\pi,t_2-t_{1}W)).$$
We are going to show that the complex of sheaves (\ref{equation_Gersten_complex}) on $\tilde{X}$ is exact at all points $x\in U_1$.

Let $x\in U_1$ be the closed point given by the intersection of the two components of the special fiber (otherwise, we are in the smooth case). We start with the \textit{global} case and let $Z'$ be a closed subscheme of~$\tilde{X}$ (the \textit{local} case $U\subset \tilde{X}$ will be deduced from this case below) of codimension $1$ which contains the point $x\in U_1$. 
As explained above, we begin by assuming that $Z'$ is flat over $\roi_K$, which is equivalent to every irreducible component of $Z'$ dominating $\Spec (\roi_K)$. 
We need to show that there is 
an open subset $U'\subset \tilde X$ containing $x$ such that the composition 
$$H_q(Z')\To H_q(\tilde{X})\To H_q(U')$$ 
is zero. By our induction assumption, we know that there is
an open subset $U\subset X$ containing $\pi_X(x)$ such that the composition $H_q(\pi_X(Z'))\to H_q(X)\to H_q(U)$ is zero. 
Note that, even though this is not necessary for the following argument, $Z'=\tilde{\pi_X(Z')}$ since $Z'$ is flat over $\Spec(\roi_K)$.
The former sequence fits into the commutative diagram 
\begin{equation}\label{diagram_blow_up_global}
  \xymatrix{ 
    & & 0 \ar[d] \\
  & & \tilde H_q(E_U) \ar[d] \\
 H_q(Z') \ar[d] \ar[r] & H_q(\tilde{X}) \ar[r] \ar[d] & H_q(\tilde{U}) \ar[d]
  \\
H_q(\pi_X(Z')) \ar[r] \ar@/_2pc/[rr]_0 & H_q(X) \ar[r] & H_q(U) \ar[d]   \\
  & & 0
  }
\end{equation}
in which $\tilde{U}$ is the blow-up of $U$ in $V(I)$ with exceptional divisor $E_U$ and in which the right vertical sequence is exact by Property~\ref{property_blowup_formula}. 

Note that since $\tilde{U}$ and $\tilde{X}$ are the blow-ups of $x\in U\subset X$, we have $E_U=E_X=\bb P^1_k$, and a calculation with the above charts shows that $E_U\cap U_1 \cong E_X\cap U_1\cong \bb A_k^{1}$. 
Let $W:=U_1\cap \tilde{U}$. Then as $\bb  A^1_k\subset W$, the assumption in the theorem implies that there exists a $W'\subset W$ such that the composition $H_q(\bb A^1_k)\to H_q(W)\to H_q(W')$ is zero. Therefore, we can extend the above commutative diagram as follows:
\[
  \xymatrix{ 
    & & 0 \ar[d] &  & \\
  & & \tilde H_q(E_U) \ar[d] \ar[r] & H_q(\bb A^1_k) \ar[d] \ar[rd]^0 & \\
 H_q(Z') \ar[d] \ar[r] & H_q(\tilde{X}) \ar[r] \ar[d] & H_q(\tilde{U}) \ar[d] \ar[r] & H_q(W) \ar[r] & H_q(W')
  \\
H_q(\pi_X(Z')) \ar[r] \ar@/_2pc/[rr]_0 & H_q(X) \ar[r] & H_q(U) \ar[d]  & & \\
  & & 0\rlap{.} & &
  }
\]
In sum, this shows that the morphism $Z'\hookrightarrow\tilde{X}$ is homologically effaceable at $x$. 

Next we show the same effaceability for any open subscheme $U$ of $\tilde{X}$ containing $x$ and any closed subscheme $Z'\subset U$ of codimension $1$ which is flat over $\Spec(\roi_K)$. The following diagram, in which $W'$ is constructed as above, commutes: 
\begin{equation}\label{diagram_local}
  \xymatrix{ 
 H_q(Z') \ar[d]_= \ar[r] \ar@/^2pc/[rr]^0 & H_q(\tilde{X}) \ar[r] \ar[d] & H_q(W') \ar[d]
  \\
H_q(Z') \ar[r]  & H_q(U) \ar[r] & H_q(U\cap W')\rlap{.} 
  }
\end{equation}
The equality of the left vertical morphism comes from the fact that the closure of $Z'$ in $\tilde{X}$ is contained in $U$ since $Z'$ is flat over $\Spec (\roi_K)$ and $\roi_K$ is henselian.\footnote{In cohomological terms, the equality on the left is the excision isomorphism $H^q_{Z'}(\tilde{X},n)\cong H^q_{Z'}(U,n)$.} The composition in the first row is zero by the above, implying the claim. This proves that the map
$$\cal H^q_{Z^{1,f}}\To \cal H^q_{Z^{0}},$$
where
\begin{align*}
  Z^{1,f}:=\{Z\subset X\mid\;& Z\text{ closed with }\codim_X(Z^c)\geq 1 \text{ and }
  \forall Z^c\subset Z \text{ irreducible},\\
 & \dim(Z^c)=\dim(\im(Z^c)\subset \Spec(\roi_K))\},
  \end{align*}
is zero. 
We note that this implies that the complex
$$ \cal H^q_{Z^{0}/Z^{1,f}}\To \bigoplus_{x\in X^{(1),f}}i_{*,x}H^{q-1}(x,n-1)\xTo{\partial_0} \bigoplus_{x\in X^{(2)}}i_{*,x}H^{q-2}(x,n-2)\To 0, $$
where $X^{(1),f}$ denotes the set of flat $1$-cycles, is exact.
We have already seen that the map
$$\cal H^q_{Z^{2}}(n)\To \cal H^q_{Z^1}(n)$$
is zero. It remains to show that the map
$$\cal H^q_{Z^{1}}(n)\To \cal H^q_{Z^0}(n)$$
is zero. This map fits into the exact sequence
$$\cdots\To\cal H^q_{Z^{1}}(n)\To \cal H^q_{Z^0}(n)\To \cal H^{q}_{Z^0/Z^{1}}(n)\To \cdots.$$
We show that the second map is injective on stalks. Let $x\in X$ be the singular point of the special fiber. Otherwise, we are in the situation of Theorem~\ref{theorem_smooth_case}. First note that $\cal H^q_{Z^0}(n)_x= H^q(\roi_{X,x},n)$ and $\cal H^q_{Z^0/Z^1}(n)_x= H^q(K(X),n)$. Let $\eta_1$ and $\eta_2$ be the two generic points of the special fiber. The first map in the exact sequence
$$\cdots\To\cal H^q_{Z^{1}-\{\eta_1,\eta_2\}}(n)\xTo{0} \cal H^q_{Z^0}(n)\To \cal H^{q}_{Z^0/(Z^{1}-\{\eta_1,\eta_2\})}(n)\To \cdots$$ 
is zero by the above. It therefore suffices to show that the map $\cal H^{q}_{Z^0/(Z^{1}-\{\eta_1,\eta_2\})}(n)_x\to H^q(K(X),n)$ is injective. But $\cal H^{q}_{Z^0/(Z^{1}-\{\eta_1,\eta_2\})}(n)_x\cong H^q(R,n)$, where $R$ is the semi-localisation at the special fiber, and in particular $R_{t_1t_2}=K(X)$. Finally, the sequence 
$$ H^q(R,n)= \cal H^{q}_{Z^0/(Z^{1}-\{\eta_1,\eta_2\})}(n)\To  \cal H^{q}_{Z^0/(Z^{1}-\{\eta_1\})}(n)\To \cal H^q_{Z^0/Z^1}(n)_x= H^q(K(X),n), $$
in which the middle term is isomorphic to $H^q(R[\frac{1}{t_2}],n)$, is a sequence of injective maps by Property~\ref{property_lifting}.
\end{proof}

\begin{remark}
There are two difficulties to overcome, and which we currently do not know how to solve, if one tries to carry out the above program for $d>2$. Firstly, if we start with $Z'$ of higher codimension and consider Diagram (\ref{diagram_blow_up_global}) with $X$ and $\tilde X$ replaced by $Y$ and $\tilde Y$ of dimension $\dim Z'+1$, then $Y$ may have singularities and the exceptional divisor of $\tilde{Y}\to Y$ may be complicated. Secondly, at least if $\dim Z'$ is different from $1$, then we cannot directly deduce the local case from the global case as in Diagram (\ref{diagram_local}) since the special fiber of the closure $\overline{Z'}\subset \tilde{X}$ may not be fully contained in $U$.
\end{remark}

\begin{remark}
The assumption on the effaceability of $\bb A^1_k\to W$ is often satisfied due to the $\bb A^1$-invariance of \'etale cohomology combined with the cohomological dimension of $k$.
\end{remark}

\begin{remark}
In \cite[Theorem 9]{Sakagaito2020} Sakagaito, using a different method, proves that if $R$ is a mixed characteristic $2$-dimensional excellent regular henselian local ring and $\ell$ a positive integer which  is invertible in $R$, then the Gersten conjecture holds for \'etale cohomology with $\mu_\ell^{\otimes n}$ coefficients. Our proof of Theorem~\ref{theorem_semi_stable} should be thought of as carrying the reduction to the dvr case in the proof of Theorem~\ref{theorem_smooth_case} from the smooth to the semi-stable situation as far as possible. 
\end{remark}


\newcommand{\etalchar}[1]{$^{#1}$}


\begin{thebibliography}{CTHK97+++}

\bibitem[AGV72]{SGA4}
M.~Artin, A.~Grothendieck and J.\,L.~Verdier (avec la collaboration de N.~Bourbaki, P.~Deligne et B.~Saint-Donat), \emph{Th\'{e}orie des topos et cohomologie \'{e}tale des sch\'{e}mas. Tome~1}, S\'{e}minaire de G\'{e}om\'{e}trie Alg\'{e}brique du Bois-Marie 1963--1964 (SGA 4), Lecture Notes in Math., vol.~269, Springer-Verlag, Berlin-New York, 1972.

\bibitem[BO74]{BO74}
S.~Bloch and A.~Ogus, \emph{Gersten's conjecture and the homology of schemes},
Ann.\ Sci.\ \'Ecole Norm.\ Sup.\ (4) \textbf{7} (1974), no.~2, 181--201,
\doi{10.24033/asens.1266}.

\bibitem[CTHK97]{CHK97}
J.\,L.~Colliot-Th\'el\`ene, R.\,T.~Hoobler and B.~Kahn, \emph{The
  {B}loch-{O}gus-{G}abber theorem}, in: \emph{Algebraic {$K$}-theory} ({T}oronto, {ON}, 1996),  pp.~31--94, Fields Inst.\ Commun., vol.~16, Amer.\ Math.\ Soc., Providence, RI, 1997,\doi{10.1090/fic/016/02}.

\bibitem[Dah18]{Da18}
C.~Dahlhausen, \emph{Milnor {K}-theory of complete discrete valuation rings
  with finite residue fields}, J.\ Pure Appl.\ Algebra \textbf{222} (2018),
  no.~6, 1355--1371, \doi{10.1016/j.jpaa.2017.07.002}.

\bibitem[Del77]{DeligneSGA}
P.~Deligne, \emph{Cohomologie \'{e}tale}, S\'{e}minaire de g\'{e}om\'{e}trie
  alg\'{e}brique du Bois-Marie (SGA $4\frac{1}{2}$), Lecture Notes in Math., vol.~569, Springer-Verlag, Berlin, 1977, \doi{10.1007/BFb0091516}.
  
\bibitem[Fuj02]{Fuj02}
K.~Fujiwara, \emph{A proof of the absolute purity conjecture (after {G}abber)},
 in: \emph{Algebraic geometry 2000}, {A}zumino ({H}otaka), pp.~153--183, Adv.\ Stud.\ Pure Math.,  vol.~36, Math.\ Soc.\ Japan, Tokyo, 2002, \doi{10.2969/aspm/03610153}.
 
\bibitem[Gab94]{Gabber1994}
O.~Gabber, \emph{Affine analog of the proper base change theorem}, Israel~J.\
  Math.\ \textbf{87} (1994), no.~1-3, 325--335,
  \doi{10.1007/BF02773001}.

\bibitem[Gei04]{Ge04}
T.~Geisser, \emph{Motivic cohomology over {D}edekind rings}, Math.~Z.\
  \textbf{248} (2004), no.~4, 773--794, \doi{10.1007/s00209-004-0680-x}.

\bibitem[GL87]{GL87}
H.~Gillet and M.~Levine, \emph{The relative form of {G}ersten's conjecture over
  a discrete valuation ring: the smooth case}, J.~Pure Appl.\ Algebra
  \textbf{46} (1987), no.~1, 59--71, \doi{10.1016/0022-4049(87)90043-0}.

  
\bibitem[Hyo88]{Hyodo1988}
O.~Hyodo, \emph{A note on {$p$}-adic \'{e}tale cohomology in the semistable
  reduction case}, Invent.\ Math.\ \textbf{91} (1988), no.~3, 543--557,
  \doi{10.1007/BF01388786}.

\bibitem[JS03]{JS03}
U.~Jannsen and S.~Saito, \emph{Kato homology of arithmetic schemes and higher
  class field theory over local fields}, Doc.\ Math.\ Extra Vol.\ (2003), 
  479--538. Kazuya Kato's fiftieth birthday, \doi{10.4171/dms/3/13}.

\bibitem[JSS14]{JSS}
U.~Jannsen, S.~Saito and K.~Sato, \emph{\'{E}tale duality for constructible
  sheaves on arithmetic schemes}, J.~reine angew.\ Math.\ \textbf{688} (2014),
  1--65, \doi{10.1515/crelle-2012-0043}.

\bibitem[Ker10]{Ke10}
M.~Kerz, \emph{Milnor {$K$}-theory of local rings with finite residue fields},
  J.~Algebraic Geom.\ \textbf{19} (2010), no.~1, 173--191,
  \doi{10.1090/S1056-3911-09-00514-1}.

\bibitem[KEW16]{KEW16}
M.~Kerz, H.~Esnault and O.~Wittenberg, \emph{A restriction isomorphism for
  cycles of relative dimension zero}, Camb.~J.\ Math.\ \textbf{4} (2016), no.~2,
  163--196, \doi{10.4310/CJM.2016.v4.n2.a1}.

\bibitem[Lau76]{Laumon1974}
G.~Laumon, \emph{Homologie \'{e}tale}, in: \emph{S\'{e}minaire de g\'{e}om\'{e}trie analytique} (\'{E}cole {N}orm.\ {S}up., {P}aris, 1974-75),  pp.~163--188,  Ast\'{e}risque \textbf{36-37}, 1976.

\bibitem[L{\"u}d24]{Lu20}
M.~L{\"u}ders, \emph{On the relative {Gersten} conjecture for {Milnor}
  {{\(K\)}}-theory in the smooth case}, J.~Pure Appl.\ Algebra \textbf{228}
  (2024), no.~11, Paper No.~107718, \doi{10.1016/j.jpaa.2024.107718}.

\bibitem[Mil80]{Mi80}
J.\,S.~Milne, \emph{\'Etale cohomology}, Princeton Math.\ Ser., vol.~33,
  Princeton Univ.\ Press, Princeton, NJ, 1980, \doi{10.2307/jj.28248894}.

\bibitem[Pan03]{Pa03}
I.\,A.~Panin, \emph{The equicharacteristic case of the {G}ersten conjecture}, Tr.\ Mat.\ Inst.\ Steklova \textbf{241} (2003), Teor.\ Chisel, Algebra i Algebr.\  Geom., 169--178.

\bibitem[Qui73]{Qu72}
D.~Quillen, \emph{Higher algebraic {$K$}-theory. {I}}, in: \emph{Algebraic {$K$}-theory, {I}: {H}igher {$K$}-theories} ({P}roc.\ {C}onf., {B}attelle {M}emorial {I}nst., {S}eattle, {W}ash., 1972), pp.~85--147. Lecture Notes in Math., vol.~341, Springer-Verlag, Berlin-New York, 1973, \doi{10.1007/BFb0067053}.

\bibitem[SS10]{SS10}
S.~Saito and K.~Sato, \emph{A finiteness theorem for zero-cycles over
  {$p$}-adic fields} (with an appendix by U.~Jannsen), Ann.\ of Math.~(2) \textbf{172} (2010), no.~3, 1593--1639, \doi{10.4007/annals.2010.172.1593}.

\bibitem[Sak20]{Sakagaito2020}
M.~Sakagaito, \emph{A note on {G}ersten's conjecture for \'{e}tale cohomology
  over two-dimensional henselian regular local rings}, C.~R.~Math.\ Acad.\ Sci.\
Paris \textbf{358} (2020), no.~1, 33--39, \href{https://doi.org/10.5802/crmath.9}{\texttt{doi:10.5802/\allowbreak crmath.9}}.

\bibitem[Sat07]{Sa07}
K.~Sato, \emph{{$p$}-adic \'etale {T}ate twists and arithmetic duality} (with an appendix by K.~Hagihara), Ann.\ Sci.\ \'Ecole Norm.\ Sup.~(4) \textbf{40} (2007), no.~4, 519--588, \doi{10.1016/j.ansens.2007.04.002}.

\bibitem[SS18]{SS18}
J.~Schmidt and F.~Strunk, \emph{Stable {$\mathbb A^1$}-connectivity over
  {D}edekind schemes}, Ann.\ K-Theory \textbf{3} (2018), no.~2, 331--367,
  \doi{10.2140/akt.2018.3.331}.

\bibitem[SS23]{SS19}
\bysame, \emph{A {Bloch}--{Ogus} theorem for henselian local rings in mixed
  characteristic}, Math.~Z.\ \textbf{303} (2023), no.~4, Paper No.~82,
  \doi{10.1007/s00209-023-03223-8}.

\bibitem[{Sta}26]{stacks-project}
The {Stacks Project Authors}, \emph{The Stacks Project},
  \url{https://stacks.math.columbia.edu}, 2026.

\bibitem[Wei13]{We13}
C.\,A.~Weibel, \emph{The {$K$}-book. An introduction to algebraic $K$-theory}, Grad.\ Stud.\ Math.,  vol.~145, Amer.\ Math.\ Soc., Providence, RI, 2013, \doi{10.1090/gsm/145}.

\end{thebibliography}
\end{document}